\documentclass{amsart}
\usepackage{graphicx,amssymb,amsmath}
\usepackage[bookmarksnumbered, plainpages]{hyperref}
\newtheorem{thm}{Theorem}[section]

\newtheorem{lem}[thm]{Lemma}
\newtheorem{prop}[thm]{Proposition}
\theoremstyle{definition}

\theoremstyle{remark}

\numberwithin{equation}{section}
\newcommand{\seq}[1]{\langle #1\rangle}
\begin{document}

\title[Groups all of whose undirected Cayley graphs are integral]{Groups all of whose undirected Cayley graphs are integral}%
\author{Alireza Abdollahi}%
\address{Department of Mathematics, University of Isfahan, Isfahan 81746-73441, Iran \\ and School of Mathematics, Institute for Research in Fundamental Sciences (IPM), P.O.Box: 19395-5746, Tehran, Iran }
\email{a.abdollahi@math.ui.ac.ir}%
\email{alireza\_abdollahi@yahoo.com}
\author{Mojtaba Jazaeri}%
\address{Department of Mathematics, University of Isfahan, Isfahan 81746-73441, Iran}
\email{seja81@sci.ui.ac.ir}
\email{seja81@gmail.com}
\thanks{This research was in part supported by a grant from IPM (No. 92050219)}%

%\date{}%
%\dedicatory{}%
%\commby{}%
% ----------------------------------------------------------------
\begin{abstract}
Let $G$ be a  finite group, $S\subseteq G\setminus\{1\}$ be a set such that if $a\in S$, then $a^{-1}\in S$, where $1$ denotes the identity element of $G$. The undirected  Cayley graph $Cay(G,S)$ of $G$ over the set $S$ is the graph whose vertex set is $G$ and two
vertices $a$ and $b$ are adjacent whenever $ab^{-1}\in S$. The adjacency  spectrum of a graph is the multiset of
 all eigenvalues of the adjacency matrix of the graph. A graph is called integral whenever all adjacency  spectrum elements are  integers.
 Following Klotz and Sander, we call a group $G$  Cayley integral whenever all undirected Cayley graphs over $G$ are integral.
Finite abelian Cayley integral  groups are classified by Klotz and Sander as finite abelian  groups of exponent dividing $4$ or $6$. Klotz and Sander have proposed  the determination of all non-abelian Cayley integral groups.
 In this paper we complete the classification of finite Cayley integral  groups by proving that  finite non-abelian Cayley integral groups are the
  symmetric group $S_{3}$ of degree $3$,  $C_{3} \rtimes C_{4}$ and  $Q_{8}\times C_{2}^{n}$ for some integer $n\geq 0$, where $Q_8$ is the quaternion group of order $8$.
\end{abstract}
\maketitle
% ----------------------------------------------------------------
\section{\bf Introduction}
Let $G$ be a finite  group and  $S$ be a subset of $G\setminus \{1\}$ such that $S=S^{-1}$,
where $1$ is the identity element of $G$. The undirected Cayley graph $Cay(G,S)$ is the graph whose vertex set is $G$ and
two vertices $a,b\in G$ are adjacent whenever $ab^{-1}\in S$. The adjacency  spectrum of a graph is the multiset of all eigenvalues of  the adjacency matrix of the graph. A graph is called integral whenever all adjacency  spectrum elements are  integers. The question of ``which graphs are integral?'' was first proposed by Harary and Schwenk \cite{HS}.
Many research papers were written on integral graphs e.g. \cite{AANS,BCRSS,So}.
Cayley graphs which are integral were studied by many people (e.g. \cite{AV1,AV,AP,DKMS,KS,KS1}).
Following \cite{KS} we call a group $G$
Cayley integral whenever all undirected Cayley graphs over $G$ are integral. Finite abelian Cayley integral groups are classified in \cite{KS}: these are finite abelian
groups of exponent dividing $4$ or $6$. In
  \cite[p. 12, Problem 3]{KS} Klotz and Sander have proposed the problem of determination of all non-abelian Cayley integral groups.
The non-abelian Cayley integral groups of order at most $12$ found in  \cite[p. 12, Problem 3]{KS} are the symmetric group $S_3$ of degree $3$, the quaternion group $Q_8$ of order $8$, and the semidirect product $C_3$ by $C_4$ which is a non-abelian group of order $12$.
In this paper we complete the classification of finite Cayley integral groups by proving  the following.
\begin{thm} \label{main}
A finite non-abelian group is Cayley integral if and only if it is isomorphic to one of the following groups:
\begin{enumerate}
 \item the symmetric group $S_{3}$ of degree $3$,
 \item  $C_{3} \rtimes C_{4}=\langle x,y \;|\; x^3=y^4=1, x^y=x^{-1} \rangle$,
\item $Q_{8}\times C_{2}^{n}$ for some integer $n\geq 0$, where $Q_8$ is the quaternion group of order $8$.
 \end{enumerate}
\end{thm}
So combining the above mentioned result of Klotz and Sander and Theorem \ref{main}, the classification of finite Cayley integral groups completes as follows.
\begin{thm}
A finite group is Cayley integral if and only if it is  isomorphic to one of the following groups:
\begin{enumerate}
\item an abelian group of exponent dividing $4$ or $6$,
 \item the symmetric group $S_{3}$ of degree $3$,
 \item  $C_{3} \rtimes C_{4}=\langle x,y \;|\; x^3=y^4=1, x^y=x^{-1} \rangle$,
\item $Q_{8}\times C_{2}^{n}$ for some integer $n\geq 0$, where $Q_8$ is the quaternion group of order $8$.
 \end{enumerate}
\end{thm}
Throughout we denote by $C_n$ the cyclic group of order $n$, the dihedral group of order $2n$ is denoted by $D_{2n}$, the alternating group of degree $4$ is denoted by $A_4$; we denote by $S_3$ and $S_4$ the symmetric groups of degree $3$ and $4$, respectively, and $C_n^k$ denotes the direct product $\underbrace{C_n\times \cdots \times C_n}_{k-\text{times}}$.
The semidirect product of a group $N$ by a group $K$ is denoted by $K\ltimes N$ or $N\rtimes K$ which is a (not necessarily unique) group $G$   containing a normal subgroup $N_1$ isomorphic to $N$ and a subgroup $K_1$ isomorphic to $K$ such that $G=N_1K_1$ and $N_1\cap K_1=1$.
 For any two elements $x,y$ of a group $G$ we denote by $[x,y]$ the commutator $x^{-1}y^{-1}xy$. For a free group $F$ generated by free generators $x_1,\dots,x_n$ and  elements $r_1,\dots,r_m$, the factor group $G=\frac{F}{\langle r_1,\dots,r_m \rangle^F}$  is denoted by
$$\langle x_1,\dots,x_n \;|\; r_1=\cdots=r_m=1\rangle, \;\;\; (*)$$ where $\langle r_1,\dots,r_m \rangle^F$ is the normal closure of the subgroup $\langle r_1,\dots,r_m \rangle$ in $F$.  The notation  $(*)$  is called the presentation of the group $G$ by generators $x_1,\dots,x_n$ and relations $r_1,\dots,r_m$.
% -----------------------------------------------------------
\section {\bf Preliminaries}
In this section we state some facts which we need in the sequel.
The following result is the classification of  all undirected connected cubic Cayley integral graphs.
\begin{thm} [Theorem 1.1 of \cite{AV}] \label{AV}
There are exactly seven connected cubic integral Cayley graphs. In particular, for a finite group $G$ and a subset $1 \notin S=S^{-1}$ with three
elements, $Cay(G,S)$ is integral if and only if $G$ is isomorphic to one of the following groups:
\begin{equation*}
C_{2}^{2}, C_{4}, C_{6}, S_{3}, C_{2}^{3}, C_{2} \times C_{4},
 D_{8}, C_{2}\times C_{6}, D_{12}, A_{4}, S_{4}, D_{8} \times C_{3}, S_{3} \times C_{4}, A_{4}\times C_{2}.
\end{equation*}
\end{thm}
\begin{lem}
Let $G$ be a finite Cayley integral group. Then every subgroup of $G$ is also Cayley integral.
\end{lem}
\begin{proof}
Let $H$ be an arbitrary subgroup of $G$ and let $T\subseteq H\setminus \{1\}$ with $T=T^{-1}$. Consider the Cayley graph $Cay(H,T)$.
Note that $Cay(H,T)$ is isomorphic to a disjoint union of some $\Gamma=Cay(\langle T \rangle, T)$. Thus $Cay(H,T)$ is integral if and only if
 $\Gamma$ is integral. Now, since $Cay(G,T)$ is also a disjoint union of some $\Gamma$, it follows that $\Gamma$ is integral.
  Hence $H$ is a Cayley integral group. This completes the proof.
\end{proof}
The above lemma will be frequently used  in the sequel without further notice.
\begin{prop}[see \cite{DS} and \cite{Ram}, Proposition 6.3.1 of \cite{BH}] \label{BH1}
Let $G$ be a finite group and $S$ a subset that is inverse closed and invariant under conjugation. The graph $Cay(G,S)$ has eigenvalues
$\theta_{\chi}=\frac{1}{\chi(1)}\sum_{s\in S}\chi(s)$ with multiplicity $\chi(1)^{2}$, where $\chi$ ranges over the irreducible characters of $G$.
\end{prop}
Note that for every finite group $G$, since  $\sum_{\chi \in \text{Irr}(G)}\chi(1)^2=|G|$ ($\text{Irr}(G)$ is the set of all irreducible characters of $G$), the multiset $\{\theta_{\chi} \;|\; \chi \in \text{Irr}(G)\}$ is the spectrum of $Cay(G,S)$ for any $S$ as in Proposition \ref{BH1}.\\
The following result has its own interest and we do not use it in the sequel but we would like to state it here!
\begin{prop}[See Theorem 2 of \cite{DKMS}]\label{DKMS} Let $G$ be a finite group and $S$ be a member of the boolean algebra generated
by the normal subgroups of $G$. Then the Cayley graph $Cay(G,S)$ is integral.
\end{prop}
\begin{proof}
Since $S$ can be obtained by arbitrary finite intersections, unions, or complements of normal subgroups of $G$, $S$ is closed under conjugation as well as  inverse. Thus by Proposition \ref{BH1} eigenvalues of $Cay(G,S)$ are of the form $\theta_{\chi}=\frac{1}{\chi(1)}\sum_{s\in S}\chi(s)$ for $\chi\in\text{Irr}(G)$. Now, Corollary 4.2 of \cite{AP} shows that $\sum_{s\in S}\chi(s)$ is  an algebraic integer for each $\chi\in\text{Irr}(G)$. Since each eigenvalue of any graph is an algebraic integer, it follows that $\theta_{\chi}$ is integer. This completes the proof.
\end{proof}
\begin{lem} [Lemma 11 of \cite{KS}] \label{KS}
If $G$ is a Cayley integral group, then the order of every non-trivial element of $G$ belongs to $\{2,3,4,6\}$.
\end{lem}
\begin{lem}\label{D8}
$D_{2n}$ is not a Cayley integral group for all integers $n\geq 4$.
\end{lem}
\begin{proof}
It is well-known that $D_{2n}=\seq{a,b|a^{n}=1, b^{2}=1, b^{-1}ab=a^{-1}}$. If $S=\{b,ba\}$, then $Cay(D_{2n},S)$ is a cycle on $2n$
vertices which is not an integral graph. Therefore $D_{2n}$ is not an integral group.
\end{proof}
\begin{lem} \label{A4}
The following groups are not Cayley integral.
\begin{enumerate}
\item The alternating group $A_4$ of degree $4$,
\item $C_{4}\rtimes C_{4}=\langle x,y \;|\; x^4=y^4=1,  x^y=x^{-1}  \rangle$,
\item $S_{3} \times C_{3}$,
\item the special linear group $SL(2,3)$ of $2\times 2$ matrices over the field of order $3$,
\item $(C_{4}\ltimes C_{3}) \times C_{2}$, where $C_{4}\ltimes C_{3}=\langle a,b \;|\; a^3=b^4=1, a^b=a^{-1} \rangle$,
\item $(C_4 \times C_2)\rtimes C_4=\langle x,y \;|\; x^4=y^4=[x,y]^2=[x^2,y]=[x,y^2]=1 \rangle$,
\item the non-abelian group of exponent $3$ and order $27$: $\langle x,y \;|\; x^3=y^3=(xy)^3=(xy^{-1})^3=1 \rangle$,
\item $(C_3\times C_3) \rtimes C_4=\langle x,y,z \;|\; x^3=y^3=z^4=[x,y]=1,x^z=x^{-1},y^z=y^{-1} \rangle$.
\end{enumerate}
\end{lem}
\begin{proof}
For each of  groups (1)-(8)  we have found an inverse closed subset $S$ $(1\not\in S)$ such that $Cay(G,S)$ is not integral.
We have used the following codes in {\sf GAP} \cite{GAP} to obtain the spectra of these graphs.
The following code constructs the Cayley graph of $A_4$ on the set $S$. We use  {\sf GRAPE} package of Soicher \cite{S}.
\begin{verbatim}
LoadPackage("grape");
F:=FreeGroup(2);
x:=F.1; y:=F.2;
G:=F/[x^3,y^2,(x*y)^3]; #G=A_4
a:=G.1;
b:=G.2;
S:=[a,a^-1,b,a*b,b^-1*a^-1];
C:=CayleyGraph(G,S);
\end{verbatim}
By the function {\sf admat} \cite[p. 6]{AV} one can construct the adjacency matrix of a given graph by {\sf GRAPE} \cite{S}.
\begin{verbatim}
A:=admat(C,12);
\end{verbatim}
The following command computes the characteristic polynomial of the adjacency matrix of $Cay(A_4,S)$ and the second command factorizes the
 polynomial into irreducible ones.
\begin{verbatim}
P:=CharacteristicPolynomial(A);
FP:=Factors(P);
\end{verbatim}
It follows that the characteristic polynomial of $Cay(A_4,S)$ is as follows:
$$(x-5)(x+1)^{5}(x^{2}-5)^{3},$$
and so $Cay(A_4,S)$ is not integral showing that $A_4$ is not a Cayley integral group.
For the other groups $G$ in (2)-(8) we give a presentation of $G$ under which we introduce a subset $S$ for which $Cay(G,S)$ is not integral.
 Verifying that $Cay(G,S)$ is not integral can be done as above by {\sf GAP}. In each case, the factorized characteristic polynomial $P(x)$ into
 irreducibles is exhibited.\\
(2) \; $G=C_{4}\rtimes C_{4}=\langle x,y \;|\; x^4=y^4=1, x^y=x^{-1}  \rangle$,  $S=\{x,x^{-1},y,y^{-1},xy,y^{-1}x^{-1},xy^2,y^{-2}x^{-1} \}$ \\ and
$P(x)=(x-8)x^{9}(x+4)^{2}(x^{2}-8)^{2}$.\\
(3) \; $G=S_{3} \times C_{3}=\langle x,y,z \;|\; x^2=y^3=z^3=[x,z]=[y,z]=1, y^x=y^{-1}\rangle$, $S=\{x,y,y^{-1},z,z^{-1},zy^2,y^{-2}z^{-1}\}$\\
 and $P(x)=(x-7)(x-5)(x-1)^{4}(x+1)^{4}(x^2+3x-1)^{4}$. \\
(4) \; $G=SL(2,3)=\langle x,y \;|\; x^3=y^4=y^{-1}xyxy^{-1}x=x^{-1}y^{-1}(x^{-1}y)^2=(xy)^3=1 \rangle$, $S=\{x,x^{-1},y,y^{-1}\}$\\ and $P(x)=(x-4)(x-2)^4(x-1)^2(x+1)^8(x+3)^3(x^2-x-4)^3$.\\
(5) \;  $G=(C_{4}\ltimes C_{3}) \times C_{2}=\langle x,y,z \;|\; x^4=y^3=z^2=[x,z]=[y,z]=1,y^x=y^{-1} \rangle$,\\
 $S=\{x,x^{-1},y,y^{-1},z,xy,y^{-1}x^{-1},xz,z^{-1}x^{-1}\}$  and  \\ $P(x)=(x-9)(x-3)^{3}(x-1)^{2}x^{6}(x+1)(x+2)^{4}(x+3)(x+4)^{2}(x^{2}-12)^{2}$.\\
(6) \; $G=(C_4 \times C_2)\rtimes C_4=\langle x,y \;|\; x^4=y^4=[x,y]^2=[x^2,y]=[x,y^2]=1 \rangle$,\\ $S=\{x,x^{-1},y,y^{-1}\}$  and
$P(x)=(x-4)(x-2)^8x^{10}(x+2)^8(x+4)(x^2-8)^2$.\\
(7) \; $G=\langle x,y \;|\; x^3=y^3=(xy)^3=(xy^{-1})^3=1 \rangle$, $S=\{x,x^{-1},y,y^{-1}\}$ and \\ $P(x)=(x-4)(x-1)^4(x+2)^{10}(x^2-2x-2)^6$.\\
(8) \; $G=(C_3\times C_3) \rtimes C_4=\langle x,y,z \;|\; x^3=y^3=z^4=[x,y]=1,x^z=x^{-1},y^z=y^{-1} \rangle$, \\ $S=\{x,x^{-1},y,y^{-1},z,z^{-1},z^2,xy,(xy)^{-1},xz,(xz)^{-1},yz,(yz)^{-1} \}$\\ and $P(x)=(x-13)(x-5)^2(x-1)^5(x+1)^{12}(x+4)^4(x^2-2x-11)^4(x^2+4x-8)^2$.
\end{proof}
\begin{lem}\label{2-x}
Let $G$ be a finite  Cayley integral group and $x$ be any element of order $2$ of $G$ and $y$ be an arbitrary element of $G$.
 Then $\langle x,y \rangle$ is isomorphic to one of the groups:
$C_2$, $C_{2}^{2}$, $C_{4}$, $C_{6}$, $S_{3}$,  $C_{2} \times C_{4}$,  $C_{2}\times C_{6}$.
\end{lem}
\begin{proof}
If $y \in \langle x\rangle$, then $\langle x,y \rangle\cong C_2$. Suppose that $y\not\in \langle x\rangle$. If $o(y)=2$,
then $\langle x,y \rangle$ is a dihedral group of order at most $6$ by Lemma \ref{D8}. Thus in the latter case, $\langle x,y \rangle\cong C_2 \times C_2$
or $S_3$. Now, assume that $o(y)>2$.
Let $S=\{x,y,y^{-1}\}$. Then $Cay(\langle S\rangle,S)$ is a cubic integral Cayley graph. It follows from Theorem \ref{AV} that $\langle S\rangle$ is
isomorphic to one of the following groups:
 $C_{4}$, $C_{6}$, $S_{3}$, $C_{2} \times C_{4}$, $D_{8}$, $C_{2}\times C_{6}$, $D_{12}$, $A_{4}$, $S_{4}$, $D_{8} \times C_{3}$, $S_{3} \times C_{4}$, $A_{4}\times C_{2}$. It follows from Lemma \ref{KS} that $\langle S \rangle$ is isomorphic to one of following groups:
$C_{4}$, $C_{6}$, $S_3$, $C_{2} \times C_{4}$, $D_{8}$, $C_{2}\times C_{6}$, $D_{12}$, $A_{4}$, $S_{4}$, $A_{4}\times C_{2}$.
Since by Lemmas \ref{D8} and \ref{A4}, $G$ cannot have any subgroup isomorphic to $D_{2n}$ ($n\geq 4$) or $A_4$, $\langle S \rangle$ is
 isomorphic to one of following groups:
$C_{4}$, $C_{6}$, $C_{2} \times C_{4}$, $C_{2}\times C_{6}$, $S_3$. This completes the proof.
\end{proof}
\begin{lem}\label{2-4-6}
Let $G$ be a finite  Cayley integral group. Let $x$ be any element of order $2$ and $y\in G$ is of order $4$ or $6$. Then $xy=yx$.
\end{lem}
\begin{proof}
It follows from Lemma \ref{2-x} that the subgroup $\langle x,y\rangle$ is abelian or isomorphic to $S_3$. Since $S_3$ has no element of order $4$ or $6$,
we are done.
\end{proof}
\begin{lem}\label{3inv}
Let $G$ be a finite non-abelian  Cayley integral group generated by three distinct elements of order $2$. Then $G\cong S_3$.
\end{lem}
\begin{proof}
Suppose that $G=\langle x,y,z \rangle$, where $x,y,z$ are all distinct and $o(x)=o(y)=o(z)=2$. Consider the cubic Cayley graph $\Gamma=Cay(G,\{x,y,z\})$. Since $G$ is non-abelian
and $\Gamma$ is integral, it follows from Theorem \ref{AV} that  $G$ is isomorphic to one of the following groups:
$S_3$, $D_{8}$, $D_{12}$, $A_{4}$, $S_{4}$, $D_{8} \times C_{3}$, $S_{3} \times C_{4}$, $A_{4}\times C_{2}$.\\
The groups $D_{8}$, $D_{12}$ and $D_{8} \times C_{3}$ are ruled out by Lemma \ref{D8} and the groups $A_{4}$, $S_{4}$ and $A_{4}\times C_{2}$ are
not possible by Lemma \ref{A4}. The group   $S_{3} \times C_{4}$  is not Cayley integral by Lemma \ref{KS}. It follows that $G\cong S_3$.
\end{proof}
\begin{lem} \label{A1}
Let $G$ be a finite $3$-group. Then $G$ is  Cayley integral if and only if $G$ is elementary abelian.
\end{lem}
\begin{proof}
If $G$ is an elementary abelian $3$-group, then $G\cong C_3^k$ for some integer $k\geq 0$. Now, it follows from \cite{KS} that $G$ is Cayley integral.\\
Now, assume that $G$ is a finite Cayley integral $3$-group.
By Lemma \ref{KS} the exponent of $G$ is $3$. Suppose, for a contradiction, that $G$ is non-abelian. Then $G$ has two non-commuting elements $x$ and $y$.
Thus $\langle x,y \rangle$ is the group of order $27$ and exponent $3$. This contradicts Lemma \ref{A4}. Thus $G$ is abelian of exponent $3$ which
means that $G$ is an elementary abelian $3$-group.
\end{proof}
\begin{thm}\label{S3}
Let $G$ be a finite  Cayley integral group. Then, there exist two non-commuting elements of order $2$ in $G$ if and only if $G\cong S_3$.
\end{thm}
\begin{proof}
Let $x,y\in G$ be two non-commuting elements of order $2$. Then it follows from Lemma \ref{2-x} that $\langle x,y \rangle\cong S_3$.
 Now, assume that $z$ is an element of order $2$ such that $z\not\in\{x,y\}$.  It follows from Lemma \ref{3inv} that $\langle x,y,z\rangle\cong S_3$.
 Therefore $z\in\langle x,y \rangle$. This means that all elements of order $2$ of $\langle x,y\rangle$ are exactly all elements of order $2$ of $G$.
  Thus $G$ has precisely  three elements $x,y,z$ of order $2$ and they are pairwise non-commuting.
Now, assume that, if possible, $G$ has an element $t$ of order $4$. Then by Lemma \ref{2-4-6} $t$ commutes with all $x$, $y$ and $z$ and so
 $t^2g=gt^2$ for all $g\in \{x,y,z\}$; this is a contradiction since $t^2\in\{x,y,z\}$.
Therefore $G$ has no element of order $4$. Now, if $G$ has a subgroup $K$ of order $4$, it must be isomorphic to $C_2 \times C_2$, a contradiction
as all elements of order $2$ of $G$ are pairwise non-commuting. Hence $4$ does not divide $|G|$ and so $|G|=2m$ for some odd integer $m$.
It follows from Lemma \ref{KS} that $m$ is a power of $3$. Let $M$ be a Sylow $3$-subgroup of $G$. By Lemma \ref{A1}  $M$ is elementary abelian.
Assume that $|M|\geq 9$. Thus $M$ has two elements $b_1$ and $b_2$ such that $\langle b_1,b_2\rangle=\langle b_1\rangle \times \langle b_2\rangle$.
Note that if $a\in G$ and $b \in G$ such that $o(a)=2$ and $o(b)=3$, it follows from Lemma \ref{2-x} that $b^a=b$ or $b^a=b^{-1}$ since $\langle a,b\rangle$ is
either abelian or isomorphic to $S_3$. Since $G$ has a subgroup isomorphic to $S_3$ (say $\langle x,y \rangle$),
 we may assume that $b_1^x=b_1^{-1}$. Thus $b_2^x=b_2$ or $b_2^x=b_2^{-1}$. If $b_2^x=b_2$, then $\langle x,b_1,b_2 \rangle\cong S_3 \times C_3$
 which is not possible by Lemma \ref{A4}. If $b_2^x=b_2^{-1}$, then $\langle x,b_1,b_2 \rangle$ has $9$ elements of order $2$
 which is a contradiction, since $G$ has only $3$ elements of order $2$.
 Thus $|M|=3$ and so $G\cong S_3$. \\
The converse is easy to verify.
\end{proof}
\begin{thm}\label{Q8}
Let $G$ be a finite non-abelian  $2$-group. Then $G$ is Cayley integral if and only if $G\cong Q_{8}\times C_{2}^{n}$ for some integer $n\geq 0$.
\end{thm}
\begin{proof}
Suppose that $G$ is a finite non-abelian Cayley integral  $2$-group. If we prove that every subgroup of $G$ is normal in $G$,  it follows from a famous result of Dedekind-Baer (see Theorem 5.3.7 of \cite{R}) that $G\cong Q_{8}\times C_{2}^{n}$ for some integer $n\geq 0$. \\
To prove that every subgroup of $G$ is normal in $G$ it is enough to show that every cyclic subgroup of $G$ is normal.
Since $G$ is a $2$-group, every element of $G$ is of order $1$, $2$ or $4$ by Lemma \ref{KS}. Thus every cyclic subgroup of $G$ is either of order $1$, $2$ or $4$.  Every element of order $2$ belongs to the center of $G$, this follows from Lemma \ref{2-x} and so every (cyclic) subgroup of order $2$ is normal in $G$. Therefore, it remains to  prove that $\langle x\rangle\trianglelefteq G$ for all elements $x$ of order $4$. \\
Suppose in contrary that there exists an element $a$ of order 4 such that $\langle a\rangle$ is not normal in $G$. Thus there exists an element $g$ of $G$ such that $g^{-1}ag \notin \seq{a}$. Consider the subgroup $H=\seq{a,g}$ of $G$. Clearly  $H$ is non-abelian. The order of $g$ is not $1$ or $2$, otherwise $a^g=a$ since elements of order $2$ are central in $G$.   Therefore the order of $g$ is $4$. Now, we investigate group properties of $H$.\\
(1) \; $H=\langle a,g\rangle$ is of exponent $4$, $o(a)=o(g)=4$ and $g^{-1}ag \notin \seq{a}$. \\
(2) \; All elements of order $2$ of $H$ are in the center $Z(H)$ of $H$. This follows from Lemma \ref{2-x}.\\
(3) \;  $H/Z(H)$ is of exponent $2$. If $x\in H$, then $x$ is of order $1,2$ or $4$ by the property (1). Thus $o(x^2)\in\{1,2\}$ and so $x^2\in Z(H)$ by the property (2). This proves that $H/Z(H)$ is of exponent $2$.\\
(4) \;   $H$ is nilpotent of class $2$. By the property (3), the derived subgroup $H'$ of $H$ is contained in $Z(H)$ and so $H$ is nilpotent of class $2$.\\
(5) \; $H'=\langle [a,g] \rangle$ is of order $2$. It is because that every commutator in $H$ is central by the property (4) and since $H$ is generated by two elements $a$ and $g$, we have $H'=\langle [a,g] \rangle$. Now, by the property (3),  $[a,g]^2=[a^2,g]=1$ and so $|H'|=2$.\\
(6) \; The factor group  $H/H'$ is isomorphic to $C_{4}\times C_{4}$, $C_{2}\times C_{4}$, or $C_{2}\times C_{2}$.
This is because $H/H'$ is an abelian group generated by two elements $aH$ and $gH$ which are both of orders dividing $4$ and
we note that $H/H'$ cannot be cyclic otherwise $H$ is also cyclic. \\
It follows that the order of $H$ is 8, 16 or 32.\\
 If the order of $H$ is $8$, then $H$ is isomorphic to $D_{8}$ or $Q_{8}$; Since by Lemma \ref{D8}, $D_{8}$ is not an integral group, $H\not \cong D_8$.
  Every subgroup of $Q_{8}$ is normal and so $g^{-1}ag \in \seq{a}$ if $H\cong Q_8$. Therefore by the property (1), $H\not\cong Q_8$.\\
  Thus $|H|\in \{16,32\}$. Now, suppose that the order of $H$ is 16. By the following code written in {\sf GAP} \cite{GAP},
  one can see what are  groups $H$ of order $16$ satisfying the properties (1)-(6).
 \begin{verbatim}
 a:=AllSmallGroups(16,IsAbelian,false);
 b:=Filtered(a,i->Size(DerivedSubgroup(i))=2);
 c:=Filtered(b,i->IdSmallGroup(FactorGroup(i,DerivedSubgroup(i)))=[8,2]);
 d:=Filtered(c,i->IsSubgroup(Center(i),Subgroup(i,Filtered(i,j->Order(j)=2)))=true);
 e:=Filtered(d,i->Exponent(i)=4);
\end{verbatim}
 The list {\tt e} contains only one group  isomorphic to $C_{4}\rtimes C_{4}=\langle x,y \;|\; x^4=y^4=1,x^y=x^{-1} \rangle$. Now, Lemma \ref{A4} implies that $H$   cannot be isomorphic to $C_{4}\rtimes C_{4}$ and so $|H|\not=16$. \\
Now, assume that $|H|=32$. By a similar code as above one can see that there is only one group $N$ satisfying the properties (1)-(6). The group $N$ is isomorphic to $(C_4 \times C_2)\ltimes C_4=\langle x,y \;|\; x^4=y^4=[x,y]^2=[x^2,y]=[x,y^2]=1 \rangle$ which is not a Cayley integral group by Lemma \ref{A4}. Thus $|H|\not=32$. This completes the proof in this direction.\\

Now, let us to prove that $T_n=Q_{8}\times C_{2}^{n}$ is a Cayley integral group for all integers $n\geq 0$.
We first prove that  the conjugacy class $a^{T_n}=\{a^g \;|\; g\in T_n\}$ of any element $a\in T_n$ is equal to $\{a\}$ or $\{a,a^{-1}\}$. For, if $o(a)=2$, then $a$ is central in $T_n$ and so $a^{T_n}=\{a\}$; and if $o(a)=4$, then $a=xt$, where $x\in Q_8$ and $t\in C_2^n$.
Let $g=ys$ be any element of $T_n$ such that $y\in Q_8$ and $s\in C_2^n$. We have $a^g=x^yt$ and since $x^y\in \{x,x^{-1}\}$ (it is easy to check that the latter is valid in $Q_8$) it follows that $a^g\in\{a,a^{-1}\}$.\\
It follows that if $S$ is an inverse closed subset of $T_n$ (not containing $1$), $S$ is also closed under conjugation of elements of $T_n$. Now, consider the Cayley graph $Cay(T_n,S)$.  Proposition \ref{BH1} implies that for each irreducible character $\chi$ of $T_n$,   $\frac{1}{\chi(1)}\sum_{s\in S}\chi(s)$
is an eigenvalue $\theta_{\chi}$ of $Cay(T_n,S)$ of multiplicity $\chi(1)^2$. As we mentioned in the paragraph after the statement of Proposition \ref{BH1} the multiset $\{\theta_{\chi} \;|\; \chi \in \text{Irr}(T_n)\}$ is the spectrum of $Cay(T_n,S)$. Now, since $\chi(g)\in \mathbb{Z}$ for all $g\in T_n$ (because all irreducible characters of $Q_8$ or $C_2$ have integer values and an irreducible character of $T_n$ is a tensor product of irreducible characters of $Q_n$ and $C_2$'s), $\sum_{s\in S}\chi(s)\in \mathbb{Z}$. Since $\chi(1)$ is an integer, it follows that $\theta_{\chi}\in\mathbb{Q}$ and so $\theta_{\chi}\in \mathbb{Z}$ as each eigenvalue of the adjacency matrix of a graph is an algebraic integer. Hence $Cay(T_n,S)$ is integral and so $T_n$ is a Cayley integral group.
\end{proof}
It should be mentioned that the subsets $S$ of $Q_8 \times C_p^{d}$ ($p$ a prime) for which the Cayley graph $Cay(Q_8 \times C_p^d,S)$ is integral are studied in the last section of \cite{DKMS}. We cannot derive the above result from  discussions in \cite{DKMS}.

% -----------------------------------------------------------
\section {\bf Proof of the main Theorem}
In this section we prove the main theorem.

{\noindent \bf Proof of Theorem \ref{main}}.\\
Let $G$ be a finite non-abelian Cayley integral group.
By Lemma \ref{A1} and Theorem \ref{Q8}, we may assume that $6$ divides the order of $G$. By Theorem \ref{S3}, we may assume that all elements of order $2$  of $G$ pairwise commute. If $x$ is an arbitrary element of order 3, and $y$ is any element of order 2, the subgroup $\seq{x,y}$ must be abelian, otherwise it is isomorphic to $S_3$ by Lemma \ref{2-x}, and $S_3$ has non-commuting elements of orders 2, a contradiction. Now, Lemmas  \ref{KS} and \ref{2-4-6} imply that  elements of order 2 lie in the center of $G$. If $G$ has no elements of order $4$, then $G=P\times Q$, where $P$ is the Sylow $2$-subgroup of $G$ and $Q$ is the Sylow $3$-subgroup of $G$; this is because $P$ is a central subgroup of $G$. Now, Lemma \ref{A1} implies that $G\cong C_2^k \times C_3^\ell$ for some positive integers $k$ and $\ell$ and so $G$ is abelian, a contradiction. Thus we may assume that $G$ has an element of order $4$.\\
Now, let $a$ and $b$ are two elements of $G$ of order $4$ and $3$, respectively. By Lemma \ref{KS}  $ab\not=ba$ and we claim that $b^a=b^{-1}$. To prove the latter, it is enough to show that $K=\langle a,b\rangle\cong C_4\ltimes C_3=\langle x,y \;|\; x^4=y^3=1,y^x=y^{-1} \rangle$. We need to note some group properties of $K$ as follows:\\
(1) \; $K$ is a  finite non-abelian Cayley integral group having two elements of orders $3$ and $4$.\\
(2) \; The set of element orders of $K$ is contained in $\{1,2,3,4,6\}$. \\
(3) \; All elements of order $2$ of $K$ lie in the center of $K$. \\
By Von Dyck's theorem, there exists an epimorphism from some $G_{i,j}$ ($i,j\in\{2,3,4,6\}$) onto $G$, where
$$G_{i,j}=\langle x,y |x^4=y^3=[x^2,y]=(xy)^i=[x,y]^j=[(xy)^{\frac{\epsilon(i)i}{2}},x]=
[(xy)^{\frac{\epsilon(i)i}{2}},y]=[[x,y]^{\frac{\epsilon(j)j}{2}},x]=[[x,y]^{\frac{\epsilon(j)j}{2}},y]=1 \rangle,$$
 and $\epsilon(\ell)=\begin{cases} 1 & \text{if} \; \ell \; \text{is even} \\ 0 & \text{if} \; \ell \; \text{is odd}\end{cases}$.
Therefore the group $K$ is isomorphic to a quotient of some $G_{i,j}$.
All groups $G_{i,j}$ are finite and can be easily computed by {\sf GAP} \cite{GAP}.
Hence we need to study  quotients of  $G_{i,j}$. Using the following code in {\sf GAP} \cite{GAP}, one can find all possible quotients (satisfying properties (1)-(3) above) of $G_{i,j}$ which can be isomorphic to $K$. Note that the following code is for $G_{4,6}$.
\begin{verbatim}
f:=FreeGroup(2);;
x:=f.1;;y:=f.2;;
G46:=f/[x^4,y^3,Comm(x^2,y),(x*y)^4,Comm(x,y)^6,Comm((x*y)^2,x), Comm((x*y)^2,y),
                         Comm(Comm(x,y)^3,x),Comm(Comm(x,y)^3,y)];;
N:=NormalSubgroups(G46);;
T:=List(N,i->FactorGroup(G46,i));;
Tn:=Filtered(T,i->IsAbelian(i)=false);;
Tn1:=Filtered(Tn,i->exponent(i)=12);;
\end{verbatim}
These possible quotients are $SL(2,3)$, $C_2 \times SL(2,3)$ or $C_{4}\ltimes C_{3}$. Now, Lemma \ref{A4} implies that $K=C_{4}\ltimes C_{3}$ as we claimed.\\
Now, let $P$ be a Sylow $2$-subgroup of $G$. Suppose for a contradiction that  $P$ is non-abelian. Then Theorem \ref{Q8} implies that $P$ has a subgroup $E$ isomorphic to the quaternion group $Q_8$ of order $8$. Consider two elements $a$ and $a'$ of order $4$ in $E$ such that $o(aa')=4$ (take famous $\mathbf{i}$, $\mathbf{j}$ and $\mathbf{k}$ in $Q_8$) and let $b$ be any element of order $3$ in $G$. Since $b^a=b^{-1}$ and $b^{a'}=b^{-1}$, it follows that $b^{aa'}=(b^{a})^{a'}=b$ and since $o(aa')=4$, $b^{aa'}=b^{-1}$. Hence $b=b^{-1}$, a contradiction. Therefore Sylow $2$-subgroups of $G$ are abelian.\\
Now, suppose for a contradiction  that $P$ has an element $a$ of order $4$ and an element $t$ of order $2$ such that $\langle a,t \rangle=\langle a\rangle \times \langle t\rangle$. Take any element $b$ of order $3$ in $G$ and consider the group $M=\langle a,b,t \rangle$. The group $M$ is isomorphic to  $(C_4\ltimes C_3)\times C_2$ which is listed in Lemma \ref{A4} as a group that is not Cayley integral. It follows that $P$ is cyclic of order $4$.
Now, let $Q$ be a Sylow $3$-subgroup of $G$. Suppose for a contradiction that $|Q|\geq 9$. Then $Q$ has two elements $b$ and $b'$ of order $3$ such that
$\langle b,b' \rangle =\langle b\rangle \times \langle b' \rangle$. Since $b^a=b^{-1}$ and ${b'}^a={b'}^{-1}$, $L=\langle b,b',a \rangle=\langle a \rangle \ltimes \langle b,b' \rangle$, it follows that the group $L$ is the group number (8) listed in Lemma \ref{A4} which is not Cayley integral. Therefore $Q$ is of order $3$ and so $G\cong \langle x,y \;|\; x^3=y^4=1, x^y=x^{-1}\rangle$. This completes  the proof. $\hfil \Box$

% -----------------------------------------------------------

\section*{\bf Acknowledgements}
 The authors thank the Graduate Studies of University of Isfahan. This research was partially supported by the Center of Excellence for Mathematics, University of Isfahan.

% ----------------------------------------------------------------

\end{document}